\documentclass{amsart}
\usepackage{amssymb}

\theoremstyle{plain}
\newtheorem{theorem}{Theorem}
\newtheorem{lemma}[theorem]{Lemma}

\theoremstyle{definition}

\theoremstyle{remark}

\newtheorem{example}[theorem]{Example}

\providecommand{\abs}[1]{\lvert#1\rvert}

\begin{document}

\title[Duality theorems]{A note on duality theorems in mass transportation}

\author{Pietro Rigo}
\address{Pietro Rigo, Dipartimento di Matematica ``F. Casorati", Universita' di Pavia, via Ferrata 1, 27100 Pavia, Italy}
\email{pietro.rigo@unipv.it}



\begin{abstract}
The duality theory of the Monge-Kantorovich transport problem is investigated in an abstract measure theoretic framework. Let $(\mathcal{X},\mathcal{F},\mu)$ and $(\mathcal{Y},\mathcal{G},\nu)$ be any probability spaces and $c:\mathcal{X}\times\mathcal{Y}\rightarrow\mathbb{R}$ a measurable cost function such that $f_1+g_1\le c\le f_2+g_2$ for some $f_1,\,f_2\in L_1(\mu)$ and $g_1,\,g_2\in L_1(\nu)$. Define $\alpha(c)=\inf_P\int c\,dP$ and $\alpha^*(c)=\sup_P\int c\,dP$, where $\inf$ and $\sup$ are over the probabilities $P$ on $\mathcal{F}\otimes\mathcal{G}$ with marginals $\mu$ and $\nu$. Some duality theorems for $\alpha(c)$ and $\alpha^*(c)$, not requiring $\mu$ or $\nu$ to be perfect, are proved. As an example, suppose $\mathcal{X}$ and $\mathcal{Y}$ are metric spaces and $\mu$ is separable. Then, duality holds for $\alpha(c)$ (for $\alpha^*(c)$) provided $c$ is upper-semicontinuous (lower-semicontinuous). Moreover, duality holds for both $\alpha(c)$ and $\alpha^*(c)$ if the maps $x\mapsto c(x,y)$ and $y\mapsto c(x,y)$ are continuous, or if $c$ is bounded and $x\mapsto c(x,y)$ is continuous. This improves the existing results in \cite{RR1995} if $c$ satisfies the quoted conditions and the cardinalities of $\mathcal{X}$ and $\mathcal{Y}$ do not exceed the continuum.
\end{abstract}

\maketitle

\section{Introduction}\label{s576bd}

Throughout, $(\mathcal{X},\mathcal{F},\mu)$ and $(\mathcal{Y},\mathcal{G},\nu)$ are probability spaces and
\begin{gather*}
\mathcal{H}=\mathcal{F}\otimes\mathcal{G}
\end{gather*}
is the product $\sigma$-field on $\mathcal{X}\times\mathcal{Y}$. Further, $\Gamma(\mu,\nu)$ is the collection of probability measures $P$ on $\mathcal{H}$ with marginals $\mu$ and $\nu$, namely,
\begin{gather*}
P(A\times\mathcal{Y})=\mu(A)\quad\text{and}\quad P(\mathcal{X}\times B)=\nu(B)\quad\text{for all }A\in\mathcal{F}\text{ and }B\in\mathcal{G}.
\end{gather*}

For any probability space $(\Omega,\mathcal{A},Q)$, we write $L_1(Q)$ to denote the class of $\mathcal{A}$-measurable and $Q$-integrable functions $\phi:\Omega\rightarrow\mathbb{R}$ (without identifying maps which agree $Q$-a.s.). We also write $Q(\phi)=\int\phi\,dQ$ for $\phi\in L_1(Q)$.

With a slight abuse of notation, for any maps $f:\mathcal{X}\rightarrow\mathbb{R}$ and $g:\mathcal{Y}\rightarrow\mathbb{R}$, we still denote by $f$ and $g$ the functions on $\mathcal{X}\times\mathcal{Y}$ given by $(x,y)\mapsto f(x)$ and $(x,y)\mapsto g(y)$. Thus, $f+g$ is the map on $\mathcal{X}\times\mathcal{Y}$ defined as
\begin{gather*}
(f+g)(x,y)=f(x)+g(y)\quad\text{for all }(x,y)\in\mathcal{X}\times\mathcal{Y}.
\end{gather*}
In this notation, we let
\begin{gather*}
L=\{f+g:f\in L_1(\mu),\,g\in L_1(\nu)\}.
\end{gather*}

Let $c:\mathcal{X}\times\mathcal{Y}\rightarrow\mathbb{R}$ be an $\mathcal{H}$-measurable function satisfying
\begin{gather}\label{new7}
f_1+g_1\le c\le f_2+g_2\quad\quad\text{for some }f_1+g_1\in L\text{ and }f_2+g_2\in L.
\end{gather}
For such a $c$, we define
\begin{gather*}
\alpha(c)=\inf\,\bigl\{P(c):\,P\in\Gamma(\mu,\nu)\bigr\},
\\\alpha^*(c)=\sup\,\bigl\{P(c):\,P\in\Gamma(\mu,\nu)\bigr\},
\\\beta(c)=\sup\,\bigl\{\mu(f)+\nu(g):\,f+ g\in L,\,f+ g\le c\bigr\},
\\\beta^*(c)=\inf\,\bigl\{\mu(f)+\nu(g):\,f+ g\in L,\,f+ g\ge c\bigr\}.
\end{gather*}
It is not hard to see that
\begin{gather*}
\beta(c)\le\alpha(c)\le\alpha^*(c)\le\beta^*(c).
\end{gather*}

A duality theorem (for both $\alpha(c)$ and $\alpha^*(c)$) is the assertion that
\begin{gather}\label{goal}
\alpha(c)=\beta(c)\quad\text{and}\quad\alpha^*(c)=\beta^*(c).
\end{gather}
Indeed, duality theorems arise in a plenty of frameworks. The main one is possibly mass transportation, where $c(x,y)$ is regarded as the cost for moving a unit of good from $x\in\mathcal{X}$ into $y\in\mathcal{Y}$. However, duality results play a role even in risk theory, optimization problems and dependence modeling. See e.g. \cite{AGS}, \cite{BESC}, \cite{BLS}, \cite{FTAP}, \cite{BPRS}, \cite{KEL}, \cite{PW}, \cite{PRWW}, \cite{LIBRO}, \cite{RUS}, \cite{VIL} and references therein.

Starting from Kantorovich himself \cite{KANT}, there is a long line of research on duality theorems; see again \cite{BESC}, \cite{BLS}, \cite{VIL} and references therein. To our knowledge, {\em under the present assumptions on} $c$, the best result is due to Ramachandran and Ruschendorf \cite{RR1995}. According to the latter, one obtains both $\alpha(c)=\beta(c)$ and $\alpha^*(c)=\beta^*(c)$ provided $c$ is $\mathcal{H}$-measurable, it satisfies condition \eqref{new7}, and at least one between $\mu$ and $\nu$ is perfect.

Now, some form of condition \eqref{new7} can not be dispensed while removing measurability leads to involve inner and outer measures; see \cite[Section 2]{KEL}. Instead, whether the perfectness assumption can be dropped is still an {\em open problem}. Thus, if $c$ is measurable and meets \eqref{new7} but $\mu$ and $\nu$ are both non-perfect, it is currently unknown whether condition \eqref{goal} is true or false. See points (2)-(3), page 355, of \cite{RR2000}.

This paper provides duality theorems not requiring perfectness.

Suppose $\mathcal{X}$ and $\mathcal{Y}$ are metric spaces and $\mathcal{F}$ and $\mathcal{G}$ the Borel $\sigma$-fields. Then, condition \eqref{goal} is shown to be true if at least one of $\mu$ and $\nu$ is separable, $c$ meets \eqref{new7} and all the $c$-sections are continuous. Or else, condition \eqref{goal} holds if $\mu$ and $\nu$ are both separable, $c$ is bounded and measurable, and at least one of the $c$-sections is continuous. These results improve \cite{RR1995} when $c$ satisfies the quoted assumptions and the cardinalities of $\mathcal{X}$ and $\mathcal{Y}$ do not exceed the continuum. Under the latter condition, in fact, a perfect probability measure is separable but not conversely. Note also that, if $\mathcal{X}$ and $\mathcal{Y}$ are separable metric spaces (so that separability of $\mu$ and $\nu$ is automatic) the scope of our results is to replace assumptions on $\mu$ or $\nu$ (required by \cite{RR1995}) with assumptions on $c$.

Various conditions for $\alpha(c)=\beta(c)$ {\em or} $\alpha^*(c)=\beta^*(c)$, but not necessarily for both, are given as well. For instance, if $c$ meets \eqref{new7} and at least one of $\mu$ and $\nu$ is separable, then $\alpha^*(c)=\beta^*(c)$  or $\alpha(c)=\beta(c)$ provided $c$ is lower or upper semicontinuous. As another example, $\alpha^*(1_H)=\beta^*(1_H)$ if $H=\cup_n (A_n\times B_n)$ with $A_n\in\mathcal{F}$ and $B_n\in\mathcal{G}$. Further, $\alpha(1_H)=\beta(1_H)$ if $\mu(\limsup_nA_n)=0$ or $\nu(\limsup_nB_n)=0$. Without some extra condition, however, we do not know whether $\alpha(1_H)=\beta(1_H)$.

\section{Preliminaries}\label{prel}

For any topological space $S$, the Borel $\sigma$-field on $S$ is denoted by $\mathcal{B}(S)$.

Let $(\Omega,\mathcal{A},Q)$ be a probability space. Then, $Q$ is {\em perfect} if, for any $\mathcal{A}$-measurable $\phi:\Omega\rightarrow\mathbb{R}$, there is $B\in\mathcal{B}(\mathbb{R})$ such that $B\subset\phi(\Omega)$ and $Q(\phi\in B)=1$.

An important special case is $\Omega$ a metric space and $\mathcal{A}=\mathcal{B}(\Omega)$. In that case, $Q$ is {\em separable} if $Q(A)=1$ for some separable $A\in\mathcal{A}$ and $Q$ is {\em tight} if $Q(A)=1$ for some $\sigma$-compact $A\in\mathcal{A}$. Clearly, tightness implies separability but not conversely. Furthermore, tightness is equivalent to perfectness provided $\Omega$ satisfies the following condition:
\begin{itemize}

\item[ ] The power set of $\Omega$ does not support any 0-1-valued probability measure $T$ such that $T\{\omega\}=0$ for each $\omega\in\Omega$;

\end{itemize}
see \cite[Theorem 3.2]{KOUM}.

Two remarks are in order. First, the above condition on $\Omega$ is automatically true if card$(\Omega)\le\,$card$(\mathbb{R})$. Thus, perfectness implies separability, but not conversely, if card$(\Omega)\le\,$card$(\mathbb{R})$ (in particular, if $\Omega$ is a separable metric space). Second, it is consistent with the usual axioms of set theory (ZFC) that, for any metric space $\Omega$, any probability measure on $\mathcal{B}(\Omega)$ is separable.

Note also that a simple example of non perfect probability measure is any non tight probability measure on the Borel sets of a separable metric space. For instance, take $Q$ the outer Lebesgue measure on $\mathcal{B}(\Omega)$, where $\Omega$ is a subset of $[0,1]$ with outer Lebesgue measure 1 and inner Lebesgue measure 0. Then, $Q$ is not perfect.

Let us come back to duality theorems. Define
\begin{gather*}
M=\bigl\{\mathcal{H}\text{-measurable functions }c:\mathcal{X}\times\mathcal{Y}\rightarrow\mathbb{R}\text{ satisfying condition \eqref{new7}}\bigr\}
\end{gather*}
and note that
\begin{gather*}
\alpha^*(c)=-\alpha(-c)\quad\text{and}\quad\beta^*(c)=-\beta(-c)\quad\text{for all }c\in M.
\end{gather*}
Thus, to get condition \eqref{goal}, it suffices to show $\alpha(c)=\beta(c)$ under some conditions which hold true for both $c$ and $-c$.

Two preliminary lemmas are needed. The first is inspired to \cite[Lemma 1]{HS}.

\begin{lemma}\label{chi55}
Let $c\in M$. Then, $\beta^*(c)=\lim_n\beta^*(c_n)$ whenever $(c_n)\subset M$ is an increasing sequence such that $c_n\uparrow c$ pointwise.
\end{lemma}

\begin{proof}
We first suppose $0\le c_n\le c\le k$ for some integer $k$. Under this assumption, for each $n$, there is $f_n+ g_n\in L$ such that
\begin{gather*}
f_n+ g_n\ge c_n,\quad 0\le f_n,\,g_n\le k,\quad \mu(f_n)+\nu(g_n)<\beta^*(c_n)+1/n;
\end{gather*}
see e.g. \cite[Lemma 1.8]{KEL}.

Since the sequences $(f_n)$ and $(g_n)$ are uniformly bounded, there are $f\in L_1(\mu)$, $g\in L_1(\nu)$ and a subsequence $(m_n)$ such that
\begin{gather*}
f_{m_n}\rightarrow f\text{ weakly in }L_1(\mu)\quad\text{and}\quad g_{m_n}\rightarrow g\text{ weakly in }L_1(\nu).
\end{gather*}
In turn, this implies the existence of a sequence $(\phi_n,\psi_n)$ such that $\phi_n\rightarrow f$ in $L_1(\mu)$, $\psi_n\rightarrow g$ in $L_1(\nu)$ and $(\phi_n,\psi_n)$ is a convex combination of $\{(f_{m_j},g_{m_j}):j\ge n\}$ for each $n$. By taking a further subsequence, it can be also assumed $\mu(\phi_n\rightarrow f)=\nu(\psi_n\rightarrow g)=1$. Since $(c_n)$ is increasing, $\phi_n+\psi_n\ge c_{m_n}$. Hence, after modifying $f$ and $g$ on null sets, one obtains $f+ g\ge c$. On noting that $(\beta^*(c_n))$ is a monotone sequence, it follows that
\begin{gather*}
\mu(f)+\nu(g)\ge\beta^*(c)\ge\lim_n\beta^*(c_n)=\lim_n\beta^*(c_{m_n})
\\=\lim_n\bigl\{\mu(f_{m_n})+\nu(g_{m_n})\bigr\}=\mu(f)+\nu(g).
\end{gather*}

This concludes the proof if $0\le c_n\le c\le k$. To deal with the general case, fix $p+q\in L$ such that $p+q\le c_1$ and define $b_n=c_n-(p+q)$ and $b=c-(p+q)$. Then, $0\le b_n\le b$. Further, since $\beta^*(h+p+q)=\beta^*(h)+\mu(p)+\nu(q)$ for each $h\in M$, it suffices to show that $\beta^*(b)=\lim_n\beta^*(b_n)$.

Given $k$, take $f_k+ g_k\in L$ such that
\begin{gather*}
f_k+ g_k\ge b\wedge 2k\quad\text{and}\quad\mu(f_k)+\nu(g_k)<\beta^*\bigl(b\wedge 2k)+1/k.
\end{gather*}
Take also $f+ g\in L$ such that $f+ g\ge b$ and note that
\begin{gather*}
f\,1_{\{g>k\}}=f\,1_{\{f\le k,g>k\}}+f\,1_{\{f>k,g>k\}}\le g\,1_{\{g>k\}}+f\,1_{\{f>k\}}.
\end{gather*}
Similarly, $g\,1_{\{f>k\}}\le g\,1_{\{g>k\}}+f\,1_{\{f>k\}}$. Hence,
\begin{gather*}
b\le b\,1_{\{b\le 2k\}}+(f+ g)\,1_{\{f+ g>2k\}}
\\\le f_k+ g_k+(f+ g)\,\bigl(1_{\{f>k\}}+1_{\{g>k\}}\bigr)
\\\le f_k+ g_k+3f\,1_{\{f>k\}}+3g\,1_{\{g>k\}}.
\end{gather*}
Since $f_k+ g_k+3f\,1_{\{f>k\}}+3g\,1_{\{g>k\}}$ belongs to $L$, it follows that
\begin{gather*}
\beta^*(b)\le\mu(f_k)+\nu(g_k)+3\mu\bigl[f\,1_{\{f>k\}}\bigr]+3\nu\bigl[g\,1_{\{g>k\}}\bigr]
\\<\beta^*(b\wedge 2k)+(1/k)+3\mu\bigl[f\,1_{\{f>k\}}\bigr]+3\nu\bigl[g\,1_{\{g>k\}}\bigr].
\end{gather*}
Fix $\epsilon>0$ and take $k$ such that $(1/k)+3\mu\bigl[f\,1_{\{f>k\}}\bigr]+3\nu\bigl[g\,1_{\{g>k\}}\bigr]<\epsilon$. By what already proved, $\beta^*(b\wedge 2k)=\lim_n\beta^*(b_n\wedge 2k)$. Therefore,
\begin{gather*}
\beta^*(b)<\beta^*(b\wedge 2k)+\epsilon=\lim_n\beta^*(b_n\wedge 2k)+\epsilon\le\lim_n\beta^*(b_n)+\epsilon.
\end{gather*}
This concludes the proof.
\end{proof}

In the second lemma, and in the rest of the paper, we write $\alpha(H)=\alpha(1_H)$ whenever $H\in\mathcal{H}$. The same notation is adopted for $\beta$, $\alpha^*$ and $\beta^*$.

\begin{lemma}\label{gh7t}
Let $c\in M$. Then, condition \eqref{goal} holds provided $\alpha(H)=\beta(H)$ for each $H\in\mathcal{H}$.
\end{lemma}

\begin{proof}

It suffices to show $\alpha(c)=\beta(c)$. To this end, we first note that $\beta(c)$ is attained, i.e., $\beta(c)=\mu(f_1)+\nu(g_1)$ for some $f_1+ g_1\in L$ such that $f_1+ g_1\le c$; see \cite[Proposition 3]{RR1995}. Define $h=c-(f_1+g_1)$ and fix $t>1$ and $P\in\Gamma(\mu,\nu)$. Then,
\begin{gather*}
P(h)=P\bigl[h\,1_{\{h\le t^{-1}\}}\bigr]+P\bigl[h\,1_{\{t^{-1}<h\le 2t\}}\bigr]+P\bigl[h\,1_{\{h>2t\}}\bigr]
\\\le t^{-1}+2t\,P(h>t^{-1})+P\bigl[h\,1_{\{h>2t\}}\bigr].
\end{gather*}
Take $f_2+ g_2\in L$ such that $f_2+ g_2\ge c$ and define
\begin{gather*}
f=f_2-f_1\quad\text{and}\quad g=g_2-g_1.
\end{gather*}
Since $h\le f+ g$,
\begin{gather*}
P\bigl[h\,1_{\{h>2t\}}\bigr]\le P\bigl[(f+g)\,1_{\{f+ g>2t\}}\bigr]\le P\bigl[(f+g)\,1_{\{f>t\}}\bigr]+P\bigl[(f+g)\,1_{\{g>t\}}\bigr]
\\=\mu\bigl[f\,1_{\{f>t\}}\bigr]+\nu\bigl[g\,1_{\{g>t\}}\bigr]+P\bigl[f\,1_{\{g>t\}}+g\,1_{\{f>t\}}\bigr].
\end{gather*}
Arguing as in the proof of Lemma \ref{chi55},
\begin{gather*}
P\bigl[f\,1_{\{g>t\}}+g\,1_{\{f>t\}}\bigr]\le 2\,P\bigl[f\,1_{\{f>t\}}+g\,1_{\{g>t\}}]=2\mu\bigl[f\,1_{\{f>t\}}\bigr]+2\nu\bigl[g\,1_{\{g>t\}}\bigr].
\end{gather*}
Hence,
\begin{gather*}
P(h)\le t^{-1}+2t\,P(h>t^{-1})+3\,\bigl\{\mu\bigl[f\,1_{\{f>t\}}\bigr]+\nu\bigl[g\,1_{\{g>t\}}\bigr]\bigr\}.
\end{gather*}

Next, by Theorem 2.1.1 and Remark 2.1.2(b) of \cite{LIBRO}, there is a {\em finitely additive} probability $Q$ on $\mathcal{H}$, with marginals $\mu$ and $\nu$, such that $Q(c)=\beta(c)$. Since $Q$ has marginals $\mu$ and $\nu$, then $\beta(H)\le Q(H)$ for all $H\in\mathcal{H}$ and
\begin{gather*}
Q(h)=Q(c)-Q(f_1+g_1)=\beta(c)-\mu(f_1)-\nu(g_1)=0.
\end{gather*}

Finally, since $h\ge 0$ and $\alpha(H)=\beta(H)$ for all $H\in\mathcal{H}$, one obtains
\begin{gather*}
\alpha(h>t^{-1})=\beta(h>t^{-1})\le Q(h>t^{-1})\le t\,Q(h)=0.
\end{gather*}
Hence, there is $P_t\in\Gamma(\mu,\nu)$ such that $P_t(h>t^{-1})<t^{-2}$. It follows that
\begin{gather*}
\alpha(c)\le P_t(c)=P_t(f_1+ g_1)+P_t(h)=\mu(f_1)+\nu(g_1)+P_t(h)
\\\le\beta(c)+3\,\bigl\{1/t+\mu\bigl[f\,1_{\{f>t\}}\bigr]+\nu\bigl[g\,1_{\{g>t\}}\bigr]\bigr\}\quad\text{for all }t>1.
\end{gather*}
Since $1/t+\mu\bigl[f\,1_{\{f>t\}}\bigr]+\nu\bigl[g\,1_{\{g>t\}}\bigr]\rightarrow 0$ as $t\rightarrow\infty$, this concludes the proof.

\end{proof}

\section{Duality theorems without perfectness}

It is convenient to distinguish two cases.

\subsection{The abstract case}

\begin{theorem}\label{sw4rf}
Let $c\in M$. Then, condition \eqref{goal} holds provided

\vspace{0.2cm}

\begin{itemize}

\item[(*)] For each $\epsilon>0$, there is a countable partition $\{A_0,A_1,\ldots\}\subset\mathcal{F}$ of $\mathcal{X}$ such that $\mu(A_0)=0$ and

\begin{gather*}
\sup_{y\in\mathcal{Y}}\,\abs{c(x,y)-c(z,y)}\le\epsilon\quad\text{whenever }x,\,z\in A_i\text{ and }i>0.
\end{gather*}

\end{itemize}

\vspace{0.2cm}

\end{theorem}

\begin{proof} Again, it suffices to show $\alpha(c)=\beta(c)$. Given $\epsilon>0$, fix a point $x_i\in A_i$ for each $i>0$, and define
\begin{gather*}
\mathcal{F}_0=\sigma\bigl(A_0\cap A,\,A_i:A\in\mathcal{F},\,i>0\bigr),\quad\mu_0=\mu|\mathcal{F}_0,
\\c_0(x,y)=1_{A_0}(x)c(x,y)+\sum_{i>0}1_{A_i}(x)c(x_i,y).
\end{gather*}
Let $\Gamma(\mu_0,\nu)$ be the set of probability measures on $\mathcal{F}_0\otimes\mathcal{G}$ with marginals $\mu_0$ and $\nu$.

Take $f_1+ g_1\in L$ and $f_2+ g_2\in L$ such that $f_1+ g_1\le c\le f_2+ g_2$. Since $\abs{c-c_0}\le\epsilon$, then $f_1+ g_1-\epsilon\le c_0\le f_2+ g_2+\epsilon$. Further, $\sup_{A_i}f_1<+\infty$ and $\inf_{A_i}f_2>-\infty$ for each $i>0$. Define
\begin{gather*}
\phi_1=-\epsilon+1_{A_0}f_1+\sum_{i>0}1_{A_i}\,\Bigl(\sup_{A_i}f_1\Bigr)\quad\text{and}\quad \phi_2=\epsilon+1_{A_0}f_2+\sum_{i>0}1_{A_i}\,\Bigl(\inf_{A_i}f_2\Bigr).
\end{gather*}
Then, $\phi_1,\,\phi_2\in L_1(\mu_0)$ and
\begin{gather*}
\phi_1+ g_1\le c_0\le \phi_2+ g_2.
\end{gather*}
Because of such inequality and since $c_0$ is $\mathcal{F}_0\otimes\mathcal{G}$-measurable, one can define
\begin{gather*}
\alpha_0=\inf_{T\in\Gamma(\mu_0,\nu)}T(c_0)\quad\text{and}\quad\beta_0=\sup_{(f,g)}\bigl[\mu_0(f)+\nu(g)]
\end{gather*}
where $\sup$ is over the pairs $(f,g)$ such that $f\in L_1(\mu_0)$, $g\in L_1(\nu)$ and $f+ g\le c_0$.

Since $\mu_0$ is an atomic probability measure, then $\mu_0$ is perfect, which in turn implies $\alpha_0=\beta_0$. Since $\abs{c-c_0}\le\epsilon$, then $\beta_0\le\beta(c)+\epsilon$. Hence, there is $T\in\Gamma(\mu_0,\nu)$ such that
\begin{gather*}
T(c_0)<\alpha_0+\epsilon=\beta_0+\epsilon\le\beta(c)+2\epsilon.
\end{gather*}
If $T$ can be extended to a probability measure $P\in\Gamma(\mu,\nu)$, then
\begin{gather*}
\alpha(c)\le P(c)\le\epsilon+P(c_0)=\epsilon+T(c_0)<3\epsilon+\beta(c).
\end{gather*}
Thus, to conclude the proof, it suffices to show that $T$ can be actually extended to a probability measure $P\in\Gamma(\mu,\nu)$.

For each $i$ with $\mu(A_i)>0$, define
\begin{gather*}
\mu_i(A)=\mu(A\mid A_i)\quad\text{ and }\quad\nu_i(B)=T\bigl(\mathcal{X}\times B\mid A_i\times\mathcal{Y}\bigr)
\end{gather*}
where $A\in\mathcal{F}$ and $B\in\mathcal{G}$. Define also
\begin{gather*}
P=\sum_i\mu(A_i)\,(\mu_i\times\nu_i),
\end{gather*}
where $\mu_i\times\nu_i$ is the product measure of $\mu_i$ and $\nu_i$ (so that $\mu_i\times\nu_i$ is a probability measure on $\mathcal{H}$). It is straightforward to see that $P\in\Gamma(\mu,\nu)$. Fix $A\in\mathcal{F}_0$ and $B\in\mathcal{G}$. For $i>0$, either $A\cap A_i=\emptyset$ or $A\cap A_i=A_i$, so that
\begin{gather*}
P(A\times B)=\sum_i\mu(A_i)\mu_i(A)\nu_i(B)=\sum_i\mu(A\mid A_i)T(A_i\times B)=T(A\times B).
\end{gather*}
Therefore, $P=T$ on $\mathcal{F}_0\otimes\mathcal{G}$.

\end{proof}

\vspace{0.2cm}

In Theorem \ref{sw4rf}, clearly, the roles of $\mu$ and $\nu$ can be interchanged. Accordingly, condition (*) can be replaced by

\begin{itemize}

\item[(**)] For each $\epsilon>0$, there is a countable partition $\{B_0,B_1,\ldots\}\subset\mathcal{G}$ of $\mathcal{Y}$ such that $\nu(B_0)=0$ and

\begin{gather*}
\sup_{x\in\mathcal{X}}\,\abs{c(x,y)-c(x,z)}\le\epsilon\quad\text{whenever }y,\,z\in B_i\text{ and }i>0.
\end{gather*}

\end{itemize}

As an example, condition (*) holds (with $A_0=\emptyset$) if $\mathcal{X}$ is a separable metric space and the function $x\mapsto c(x,y)$ is Lipschitz uniformly with respect to $y$, i.e.,
\begin{gather}\label{p998b}
\sup_{y\in\mathcal{Y}}\,\abs{c(x,y)-c(z,y)}\le u\,d(x,z)\quad\text{for all }x,\,z\in\mathcal{X}
\end{gather}
where $u>0$ is a constant and $d$ the distance on $\mathcal{X}$. Fix in fact $\epsilon>0$. Because of separability, $\mathcal{X}$ can be partitioned into sets $A_1,A_2,\ldots$ whose diameter is less than $\epsilon/u$. Hence, condition (*) follows trivially from \eqref{p998b}. Similarly, condition (**) holds if $\mathcal{Y}$ is a separable metric space and $y\mapsto c(x,y)$ is Lipschitz uniformly with respect to $x$. Further, as shown in the proof of Theorem \ref{v6y}, separability of $\mathcal{X}$ (of $\mathcal{Y}$) can be weakened into separability of $\mu$ (of $\nu$).

Another example is the following. Let $\mathcal{R}$ be the field of subsets of $\mathcal{X}\times\mathcal{Y}$ generated by the measurable rectangles. Each $R\in\mathcal{R}$ can be written as $R=\cup_{i=1}^n(A_i\times B_i)$ for some $n\ge 1$ and $A_i\in\mathcal{F}$, $B_i\in\mathcal{G}$ such that $A_i\cap A_j=\emptyset$ for $i\neq j$. Thus, when $c=1_R$ with $R\in\mathcal{R}$, condition (*) is trivially true and Theorem \ref{sw4rf} yields $\alpha(R)=\beta(R)$ and $\alpha^*(R)=\beta^*(R)$. We next prove duality of certain sets related to $\mathcal{R}$.

\begin{theorem}\label{z19o}
Let $H=\cup_nR_n$ and $K=\cap_nR_n$ where $R_n\in\mathcal{R}$ for each $n$. Then,
\begin{gather*}
\alpha^*(H)=\beta^*(H)\quad\text{and}\quad\alpha(K)=\beta(K).
\end{gather*}
In addition, $\alpha(H)=\beta(H)$ provided $H$ can be written as $H=\cup_n (A_n\times B_n)$ with $A_n\in\mathcal{F}$, $B_n\in\mathcal{G}$, and
\begin{gather*}
\mu(\limsup_nA_n)=0\quad\text{or}\quad\nu(\limsup_nB_n)=0.
\end{gather*}
(Here, $\limsup_nA_n=\cap_n\cup_{j>n}A_j$ and $\limsup_nB_n=\cap_n\cup_{j>n}B_j$).
\end{theorem}

\begin{proof}
Let $H_n=\cup_{i=1}^nR_i$. Since $\alpha^*(H_n)=\beta^*(H_n)$, Lemma \ref{chi55} implies
\begin{gather*}
\alpha^*(H)=\sup_{P\in\Gamma(\mu,\nu)}P(H)=\sup_{P\in\Gamma(\mu,\nu)}\sup_nP(H_n)=\sup_n\sup_{P\in\Gamma(\mu,\nu)}P(H_n)
\\=\sup_n\alpha^*(H_n)=\sup_n\beta^*(H_n)=\beta^*(H).
\end{gather*}
Thus, $\alpha^*$ and $\beta^*$ agree on countable unions of elements of $\mathcal{R}$. Since $R_n^c\in\mathcal{R}$, this implies
\begin{gather*}
\alpha(K)=1-\alpha^*(K^c)=1-\alpha^*\bigl(\cup_nR_n^c\bigr)
\\=1-\beta^*\bigl(\cup_nR_n^c\bigr)=1-\beta^*(K^c)=\beta(K).
\end{gather*}
Next, suppose $H=\cup_n (A_n\times B_n)$ and $\mu(\limsup_nA_n)=0$. Let $V_n=\cup_{i=1}^n(A_i\times B_i)$. Given $\epsilon>0$, take $n\ge 1$ such that $\mu\bigl(\cup_{i>n}A_i\bigr)<\epsilon$, and then take $P\in\Gamma(\mu,\nu)$ satisfying $P(V_n)<\alpha(V_n)+\epsilon$. Since $\alpha(V_n)=\beta(V_n)$, one obtains
\begin{gather*}
\alpha(H)\le P(H)\le P(V_n)+P\bigl(\cup_{i>n}(A_i\times B_i)\bigr)\le P(V_n)+\mu\bigl(\cup_{i>n}A_i\bigr)
\\<\alpha(V_n)+2\epsilon=\beta(V_n)+2\epsilon\le\beta(H)+2\epsilon.
\end{gather*}
The proof is exactly the same if $\nu(\limsup_nB_n)=0$.
\end{proof}

Because of Theorem \ref{z19o}, a (classical) question raised by Arveson \cite{ARV} admits a positive answer for countable unions of measurable rectangles.

\vspace{0.2cm}

\noindent {\bf Arveson's problem:} {\em If $H\in\mathcal{H}$ satisfies $P(H)=0$ for all $P\in\Gamma(\mu,\nu)$, are there $A\in\mathcal{F}$ and $B\in\mathcal{G}$ such that $\mu(A)=\nu(B)=0$ and $H\subset (A\times\mathcal{Y})\cup (\mathcal{X}\times B)$ ?}

\vspace{0.2cm}

\noindent Indeed, it is not hard to see that $\beta^*(H)=\mu(A)+\nu(B)$ for some $A\in\mathcal{F}$ and $B\in\mathcal{G}$ with $H\subset (A\times\mathcal{Y})\cup (\mathcal{X}\times B)$; see e.g. \cite[Lemma 1]{HS}. If $H$ is a countable union of measurable rectangles, Theorem \ref{z19o} implies $\beta^*(H)=\alpha^*(H)=0$ so that $\mu(A)=\nu(B)=0$.

In addition, exploiting Theorem \ref{z19o}, duality for $H=\cup_n (A_n\times B_n)$ can be obtained under various conditions. One such condition is $\mu(\limsup_nA_n)=0$ or $\nu(\limsup_nB_n)=0$. A similar condition is that $H^c$ is also a countable union of measurable rectangles. In this case, in fact, $\alpha^*(H^c)=\beta^*(H^c)$ or equivalently $\alpha(H)=\beta(H)$. A last condition is

\begin{gather}\label{vnbg67r}
\text{for each }n\ge 1\text{ there is a measurable function }\phi_n:\mathcal{X}\rightarrow\mathcal{Y}\text{ such that}
\\\nu=\mu\circ\phi_n^{-1}\quad\text{and}\quad\mu\bigl\{x:(x,\phi_n(x))\in H\bigr\}< 1/n.\notag
\end{gather}

\noindent Define in fact $P_n(U)=\mu\bigl\{x:(x,\phi_n(x))\in U\bigr\}$ for each $n\ge 1$ and $U\in\mathcal{H}$. Then, $P_n\in\Gamma(\mu,\nu)$ and $\alpha(H)\le P_n(H)< 1/n$. Thus $\alpha(H)=0$, which in turn implies $\alpha(H)=\beta(H)$. Here is a simple example.

\begin{example}\label{v8j9}
Suppose $(\mathcal{X},\mathcal{F})=(\mathcal{Y},\mathcal{G})$ and $\mu=\nu$, with $\mathcal{X}$ a separable metric space and $\mathcal{F}=\mathcal{B}(\mathcal{X})$. (Up to some technicalities, separability of $\mathcal{X}$ could be weakened into separability of $\mu$). Let $\Delta=\{(x,x):x\in\mathcal{X}\}$ be the diagonal and $H$ a countable union of measurable rectangles. Then, {\em duality holds for $H\cap\Delta^c$, and it holds for $H\cap\Delta$ provided $\mu$ vanishes on singletons}. In fact, $H\cap\Delta^c$ is a countable union of measurable rectangles and $\mu\bigl\{x:(x,x)\in H\cap\Delta^c\bigr\}=0$. Letting $\phi_n(x)=x$, Theorem \ref{z19o} and condition \eqref{vnbg67r} yield
\begin{gather*}
\alpha(H\cap\Delta^c)=\beta(H\cap\Delta^c)\quad\text{and}\quad\alpha^*(H\cap\Delta^c)=\beta^*(H\cap\Delta^c).
\end{gather*}
To deal with $H\cap\Delta$, suppose $\mu$ null on singletons and define $P_1=\mu\times\mu$ and $P_2(U)=\mu\bigl\{x:(x,x)\in U\bigr\}$ for each $U\in\mathcal{H}$. Then, $P_1,\,P_2\in\Gamma(\mu,\mu)$. Since $\mu$ is null on singletons, $\alpha(H\cap\Delta)\le P_1(H\cap\Delta)\le P_1(\Delta)=0$, which in turn implies $\alpha(H\cap\Delta)=\beta(H\cap\Delta)$. Finally, writing $H$ as $H=\cup_n(A_n\times B_n)$, one obtains
\begin{gather*}
\alpha^*(H\cap\Delta)\le\beta^*(H\cap\Delta)\le\mu\bigl(\cup_n(A_n\cap B_n)\bigr)=P_2(H\cap\Delta)\le\alpha^*(H\cap\Delta).
\end{gather*}
\end{example}

\vspace{0.2cm}

We close this Subsection with two remarks. The first (stated as a lemma) suggests a possible strategy for proving a general duality theorem.

\begin{lemma}
Let $\mathcal{H}_0=\bigl\{H\in\mathcal{H}:\alpha(H)=\beta(H)$ and $\alpha^*(H)=\beta^*(H)\bigr\}$. Then, condition \eqref{goal} holds for each $c\in M$ if and only if
\begin{gather}\label{bbt67}
H_n\in\mathcal{H}_0\text{ and }H_n\subset H_{n+1}\text{ for each }n\quad\Longrightarrow\quad\alpha\bigl(\cup_nH_n\bigr)=\beta(\cup_nH_n\bigr).
\end{gather}
\end{lemma}

\begin{proof}
By Lemma \ref{gh7t}, it suffices to show that $\mathcal{H}_0=\mathcal{H}$. In turn, since $\mathcal{R}\subset\mathcal{H}_0$, it suffices to see that $\mathcal{H}_0$ is a monotone class. Also, since $\mathcal{H}_0$ is closed under complements, it is enough to prove that $H\in\mathcal{H}_0$ provided $H$ is the union of an increasing sequence of elements of $\mathcal{H}_0$. Let  $H=\cup_n H_n$ where $H_n\in\mathcal{H}_0$ and $H_n\subset H_{n+1}$ for each $n$. For such $H$, arguing as in the proof of Theorem \ref{z19o}, one obtains $\alpha^*(H)=\beta^*(H)$. Thus, under \eqref{bbt67}, $\mathcal{H}_0$ is actually a monotone class.
\end{proof}

The second remark briefly compares the arguments underlying Theorem \ref{sw4rf} and the usual duality theorems. The latter are summarized into the result by Ramachandran and Ruschendorf \cite{RR1995}.

For definiteness, we aim to prove $\alpha(c)=\beta(c)$. By \eqref{new7} and since $\beta(c)$ is attained, it can be assumed $c\ge 0$ and $\beta(c)=0$. As noted in the proof of Lemma \ref{gh7t}, there is a finitely additive probability $Q$ on $\mathcal{H}$, with marginals $\mu$ and $\nu$, satisfying $Q(c)=\beta(c)$. Since $c\ge 0$ and $\beta(c)=0$, it must be $Q(c>\epsilon)=0$ for each $\epsilon>0$. A basic intuition in \cite{RR1995} is that, if one of $\mu$ and $\nu$ is perfect, then $Q$ is $\sigma$-additive on $\mathcal{R}$; see \cite[Theorem 2.1.3]{LIBRO} and recall that $\mathcal{R}$ is the field generated by the measurable rectangles. Hence, there is $P\in\Gamma(\mu,\nu)$ such that $P=Q$ on $\mathcal{R}$. With such a $P$, one obtains
\begin{gather*}
P\bigl(\cup_i R_i\bigr)=\sup_n P\bigl(\cup_{i=1}^n R_i\bigr)=\sup_n Q\bigl(\cup_{i=1}^n R_i\bigr)\le Q\bigl(\cup_i R_i\bigr)
\end{gather*}
provided $R_i\in\mathcal{R}$ for all $i$. Hence, $P(c>\epsilon)\le Q(c>\epsilon)=0$ if the set $\{c>\epsilon\}$ is a countable union of measurable rectangles. Up to some technicalities, suitable versions of this argument work even if $\{c>\epsilon\}$ fails to be a countable union of measurable rectangles. This provides a rough sketch of the proof of $\alpha(c)=\beta(c)$ under the assumption that one of $\mu$ and $\nu$ is perfect. We now turn to Theorem \ref{sw4rf}. Here, instead of proving that $Q$ is $\sigma$-additive on $\mathcal{R}$, one requires that $c$ can be suitably approximated by $\mathcal{R}$-simple functions. For instance, conditions (*)-(**) are trivially true if $c$ is the uniform limit of a sequence of $\mathcal{R}$-simple functions, and in this case no assumptions on $\mu$ or $\nu$ are needed. Apparently, conditions (*)-(**) are too restrictive to be useful in real problems. Instead, they allow to get duality in various situations, including Theorem \ref{z19o}, Example \ref{v8j9}, and the results in the next subsection.

\subsection{The metric case}\label{a3q}

In this subsection, $\mathcal{X}$ and $\mathcal{Y}$ are metric spaces, $\mathcal{F}=\mathcal{B}(\mathcal{X})$ and $\mathcal{G}=\mathcal{B}(\mathcal{Y})$. The {\em sections of $c$} are the functions $x\mapsto c(x,y)$ and $y\mapsto c(x,y)$, with $y$ fixed in the first map and $x$ fixed in the second.

A remark is in order. All claims made so far are still valid, even if $c$ is not $\mathcal{H}$-measurable, provided $c\,1_{A\times B}$ is $\mathcal{H}$-measurable for some $A\in\mathcal{F}$ and $B\in\mathcal{G}$ with $\mu(A)=\nu(B)=1$. In fact, $\alpha(c)=\alpha(c\,1_{A\times B})$ whenever $\alpha(c)$ is defined in the obvious way, i.e.
\begin{gather*}
\alpha(c)=\inf_{P\in\Gamma(\mu,\nu)}\overline{P}(c)\quad\text{where }\overline{P}\text{ is the completion of }P.
\end{gather*}
Similarly, $\alpha^*(c)=\alpha^*(c\,1_{A\times B})$, $\beta(c)=\beta(c\,1_{A\times B})$ and $\beta^*(c)=\beta^*(c\,1_{A\times B})$.

In the next result, $c\,1_{A\times B}$ is actually $\mathcal{H}$-measurable for some $A\in\mathcal{F}$ and $B\in\mathcal{G}$ such that $\mu(A)=\nu(B)=1$ (with possibly $A=\mathcal{X}$ or $B=\mathcal{Y}$).

\begin{theorem}\label{v6y}
Suppose $c$ satisfies condition \eqref{new7}, the map $x\mapsto c(x,y)$ is continuous for each $y\in\mathcal{Y}$ and the map $y\mapsto c(x,y)$ is $\mathcal{G}$-measurable for each $x\in\mathcal{X}$. Then,

\begin{itemize}

\item[(i)] $\alpha^*(c)=\beta^*(c)$ if $c$ is bounded below and $\mu$ is separable;

\item[(ii)] $\alpha(c)=\beta(c)$ if $c$ is bounded above and $\mu$ is separable;

\item[(iii)] $\alpha(c)=\beta(c)$ and $\alpha^*(c)=\beta^*(c)$ if $c$ is bounded and $\mu$ is separable;

\item[(iv)] $\alpha(c)=\beta(c)$ and $\alpha^*(c)=\beta^*(c)$ if all the sections of $c$ are continuous and at least one of $\mu$ and $\nu$ is separable.

\end{itemize}

\end{theorem}

\begin{proof}
Since (ii) and (iii) are consequences of (i), it suffices to prove (i) and (iv).

Let $\mu$ and $c$ be as in (i). Since $\mu$ is separable, there is a separable set $A\in\mathcal{F}$ with $\mu(A)=1$. Since $x\mapsto c(x,y)$ is continuous, $y\mapsto c(x,y)$ is Borel measurable and $A$ is separable, the restriction of $c$ on $A\times\mathcal{Y}$ is measurable with respect to $\mathcal{B}(A)\otimes\mathcal{B}(\mathcal{Y})$. Therefore, $c\,1_A$ is $\mathcal{H}$-measurable.

Take a countable set $D\subset A$ such that $\overline{D}=\overline{A}$ and define
\begin{gather*}
c_n(x,y)=\inf_{z\in D}\bigl\{n\,d(x,z)+c(z,y)\bigr\}
\end{gather*}
where $n\ge 1$, $(x,y)\in\mathcal{X}\times\mathcal{Y}$ and $d$ is the distance on $\mathcal{X}$.

Since $c$ is bounded below, $c_n$ is real-valued, and a direct calculation shows that
\begin{gather}\label{pij8y7}
\sup_{y\in\mathcal{Y}}\,\abs{c_n(x,y)-c_n(z,y)}\le n\,d(x,z)\,\,\,\text{ for all }x,\,z\in\mathcal{X}.
\end{gather}
Since $D$ is countable, $y\mapsto c_n(x,y)$ is Borel measurable. Hence, $c_n\,1_A$ is $\mathcal{H}$-measurable. In addition, $c_n\le c_{n+1}$ and $c_n$ meets condition \eqref{new7} (since $c$ meets \eqref{new7} and is bounded below). Finally, since $x\mapsto c(x,y)$ is continuous, one obtains
\begin{gather*}
c(x,y)=\sup_nc_n(x,y)\quad\text{for all }(x,y)\in\overline{A}\times\mathcal{Y}.
\end{gather*}

Next, $\overline{D}=\overline{A}$ implies $A\subset\bigcup_{x\in D}B(x,\delta)$ for each $\delta>0$, where $B(x,\delta)$ is the $\mathcal{X}$-ball of radius $\delta$ around $x$. Given $n\ge 1$ and $\epsilon>0$, it follows that $A$ can be partitioned into sets $A_1,A_2,\ldots\in\mathcal{F}$ whose diameter is less than $\epsilon/n$. Hence, $c_n$ meets condition (*) (with $A_0=A^c$) because of \eqref{pij8y7}. By Lemma \ref{chi55} and Theorem \ref{sw4rf},
\begin{gather*}
\alpha^*(c)=\alpha^*(c\,1_A)\le\beta^*(c\,1_A)=\lim_n\beta^*(c_n\,1_A)
\\=\lim_n\alpha^*(c_n\,1_A)\le\alpha^*(c\,1_A)=\alpha^*(c).
\end{gather*}
This concludes the proof of (i).

Let us turn to (iv). Suppose that all the $c$-sections are continuous. Since the sections of $-c$ are continuous as well, it suffices to prove $\alpha^*(c)=\beta^*(c)$. We first assume $\mu$ separable.

By \eqref{new7}, there are $\psi\in L_1(\mu)$ and $g\in L_1(\nu)$ such that $\psi+g\le c$. Define
\begin{gather*}
f(x)=\inf_{y\in\mathcal{Y}}\bigl\{c(x,y)-g(y)\bigr\},\quad x\in\mathcal{X},
\end{gather*}
and note that $f\in L_1(\mu)$, $f+g\le c$ and $f$ is upper-semicontinuous. Define also
\begin{gather*}
c_n(x,y)=\inf_{z\in\mathcal{X}}\bigl\{n\,d(x,z)+c(z,y)-f(z)\bigr\}.
\end{gather*}

Again, condition \eqref{pij8y7} holds, $c_n$ meets condition \eqref{new7} (since $g\le c_n\le c-f$) and $y\mapsto c_n(x,y)$ is Borel measurable (it is in fact upper-semicontinuous). On noting that $x\mapsto c(x,y)-f(x)$ is lower-semicontinuous, it is not hard to see that $c_n\uparrow c-f$ pointwise as $n\rightarrow\infty$. Because of \eqref{pij8y7} and $\mu$ separable, $c_n$ meets condition (*). Thus,
\begin{gather*}
\beta^*(c)-\mu(f)=\beta^*(c-f)=\lim_n\beta^*(c_n)
\\=\lim_n\alpha^*(c_n)\le\alpha^*(c-f)=\alpha^*(c)-\mu(f).
\end{gather*}
Hence, $\alpha^*(c)=\beta^*(c)$ if $\mu$ is separable.

Finally, if $\nu$ is separable, it suffices to let
\begin{gather*}
c_n(x,y)=\inf_{z\in\mathcal{Y}}\bigl\{n\,\rho(y,z)+c(x,z)-g(z)\bigr\}
\end{gather*}
where now $g$ is upper-semicontinuous and $\rho$ is the distance on $\mathcal{Y}$. Arguing as above and using separability of $\nu$, it follows that $c_n$ meets condition (**) and $c_n\uparrow c-g$ pointwise as $n\rightarrow\infty$. Hence, $\alpha^*(c)=\beta^*(c)$ and this concludes the proof.
\end{proof}

\label{0.2cm}

Once again, the roles of $\mu$ and $\nu$ can be interchanged in Theorem \ref{v6y}.

\begin{theorem}\label{h8u9i}
Suppose $c$ satisfies condition \eqref{new7}, the map $x\mapsto c(x,y)$ is $\mathcal{F}$-measurable for each $y\in\mathcal{Y}$ and the map $y\mapsto c(x,y)$ is continuous for each $x\in\mathcal{X}$. Then,

\begin{itemize}

\item[(j)] $\alpha^*(c)=\beta^*(c)$ if $c$ is bounded below and $\nu$ is separable;

\item[(jj)] $\alpha(c)=\beta(c)$ if $c$ is bounded above and $\nu$ is separable;

\item[(jjj)] $\alpha(c)=\beta(c)$ and $\alpha^*(c)=\beta^*(c)$ if $c$ is bounded and $\nu$ is separable.

\end{itemize}

\end{theorem}

\vspace{0.2cm}

Note that, if $\mu$ and $\nu$ are both separable, then $\alpha(c)=\beta(c)$ and $\alpha^*(c)=\beta^*(c)$ provided $c$ is bounded, $\mathcal{H}$-measurable, and at least one of the $c$-sections is continuous. Further, the argument underlying Theorems \ref{v6y}-\ref{h8u9i} yields other similar results. As an example, we state (without a proof) the following.

\begin{theorem}\label{a55v7}
Suppose $c$ satisfies condition \eqref{new7} and at least one of $\mu$ and $\nu$ is separable. Then, $\alpha^*(c)=\beta^*(c)$ if $c$ is lower-semicontinuous (with respect to the product topology on $\mathcal{X}\times\mathcal{Y}$) and $\alpha(c)=\beta(c)$ if $c$ is upper-semicontinuous.
\end{theorem}

Finally, we list some consequences of Theorems \ref{v6y}-\ref{a55v7}. Indeed,  they unify and slightly improve some known results.

\vspace{0.2cm}

\begin{itemize}

\item Theorems \ref{v6y}-\ref{h8u9i} improve \cite{RR1995}, the result by Ramachandran and Ruschendorf, provided $c$ satisfies some conditions and
\begin{gather*}
\text{card}(\mathcal{X})\le\text{card}(\mathbb{R})\quad\text{and}\quad\text{card}(\mathcal{Y})\le\text{card}(\mathbb{R}).
\end{gather*}
Under such cardinality assumption, in fact, perfectness implies separability but not conversely; see Section \ref{prel}.

\item As an example, suppose $c\in M$ and $\mathcal{X}$ and $\mathcal{Y}$ are separable metric spaces (so that $\mu$ and $\nu$ are both separable and the cardinality assumption is satisfied). Then, \cite{RR1995} implies $\alpha(c)=\beta(c)$ and $\alpha^*(c)=\beta^*(c)$ provided at least one between $\mu$ and $\nu$ is perfect. Instead, Theorems \ref{v6y}-\ref{h8u9i} lead to the same conclusions whenever all the $c$-sections are continuous, or whenever $c$ is bounded and at least one of the $c$-sections is continuous.

\vspace{0.2cm}

\item By Theorems \ref{v6y}-\ref{h8u9i}, it is consistent with the usual axioms of set theory (ZFC) that condition \eqref{goal} holds for every $c\in M$ with continuous sections, or for every bounded $c\in M$ with at least one continuous section. In fact, as noted in Section \ref{prel}, it is consistent with ZFC that any Borel probability on any metric space is separable.

\vspace{0.2cm}

\item Let $\mathcal{X}=\mathcal{Y}$ and $c=d$, where $d$ is the distance on $\mathcal{X}$. Suppose $d$ measurable with respect to $\mathcal{B}(\mathcal{X})\otimes\mathcal{B}(\mathcal{X})$ and
\begin{gather*}
\int d(x,x_0)\,\mu(dx)+\int d(x,x_0)\,\nu(dx)<\infty\quad\text{for some }x_0\in\mathcal{X}.
\end{gather*}
Then, $\alpha(d)$ reduces to Wasserstein distance between $\mu$ and $\nu$ while $\beta(d)$ can be written as
\begin{gather*}
\beta(d)=\sup_f\,\abs{\mu(f)-\nu(f)}
\end{gather*}
where $\sup$ is over the 1-Lipschitz functions $f:\mathcal{X}\rightarrow\mathbb{R}$. In this case, it is well known that $\alpha(d)=\beta(d)$ if $\mathcal{X}$ is separable; see e.g. \cite[page 400]{KEL}. This known fact is generalized by Theorems \ref{v6y}-\ref{a55v7} under two respects: separability of $\mathcal{X}$ can be weakened into separability of at least one of $\mu$ and $\nu$, and $d$ can be replaced by any upper-semicontinuous function or by any function with continuous sections.

\vspace{0.2cm}

\item By Theorem \ref{a55v7}, Arveson's question has a positive answer if $H$ is open and one of $\mu$ and $\nu$ is separable.

\end{itemize}

\end{document}